\documentclass[11pt]{article}

\usepackage{amsmath,amsfonts,amssymb}
\usepackage{amsthm}
\usepackage{mathrsfs}
\usepackage{enumerate}
\usepackage{fullpage}
\usepackage{hyperref}
\usepackage{pdfsync}
\usepackage{pdfpages}
\usepackage{graphicx}
\usepackage{float}
\usepackage{cleveref}

\newcommand{\N}{\ensuremath{\mathbb{N}}}

\theoremstyle{plain}           	
\newtheorem{theorem}{Theorem}[section]
\newtheorem{lemma}[theorem]{Lemma}
\newtheorem{corollary}[theorem]{Corollary}
\newtheorem{definition}[theorem]{Definition}
\newtheorem{conjecture}[theorem]{Conjecture}

\newtheorem*{claim}{Claim}

\theoremstyle{remark}
\newtheorem{rem}[theorem]{Remark}

\title{Ramsey goodness of bounded degree trees}
\date{}
\author{Igor Balla  \thanks{Department of Mathematics, ETH, 8092 Zurich. igor.balla@math.ethz.ch.} \and
Alexey Pokrovskiy \thanks{Department of Mathematics, ETH, 8092 Zurich. dr.alexey.pokrovskiy@gmail.com.} \and
Benny Sudakov \thanks{Department of Mathematics, ETH, 8092 Zurich. benjamin.sudakov@math.ethz.ch. Research supported in part by SNSF grant 200021-149111.}}

\begin{document}

\maketitle 

\begin{abstract}
Given a pair of graphs $G$ and $H$, the Ramsey number $R(G,H)$ is the smallest $N$ such that every red-blue coloring of the edges of the complete graph $K_N$ contains a red copy of $G$ or a blue copy of $H$. If a graph $G$ is connected, it is well known and easy to show that $R(G,H) \geq (|G|-1)(\chi(H)-1)+\sigma(H)$, where $\chi(H)$ is the chromatic number of $H$ and $\sigma(H)$ is the size of the smallest color class in a $\chi(H)$-coloring of $H$. A  graph $G$ is called $H$-good if $R(G,H)= (|G|-1)(\chi(H)-1)+\sigma(H)$. The notion of Ramsey goodness was introduced by Burr and Erd\H{o}s in 1983 and has been extensively studied since then.

In this paper we show that  if $n\geq \Omega(|H| \log^4 |H|)$ then every $n$-vertex bounded degree tree $T$ is $H$-good. 
The dependency between $n$ and $|H|$ is tight up to $\log$ factors.
This substantially improves a result of Erd\H{o}s, Faudree, Rousseau, and Schelp from 1985, who proved that $n$-vertex bounded degree trees are $H$-good when when $n \geq \Omega(|H|^4)$.
\end{abstract}

\section{Introduction}
For a pair of graphs $G$ and $H$, the Ramsey number $R(G, H)$ is defined to be the minimum $N$ such that every red-blue coloring of the edges of the complete graph $K_N$ contains a red copy of $G$ or a blue copy of $H$. An old theorem of Ramsey states that $R(K_n, K_n)$ is finite and therefore $R(G, H)$ is well-defined for any $G, H$. It is sometimes quite difficult to compute the Ramsey number. Indeed, the inequalities 
\[ 2^{n/2} \leq R(K_n, K_n) \leq 4^n\]
were proven by Erd\H{o}s and Szekeres \cite{ErdSze35} in 1935, and Erd\H{o}s \cite{Erd47} in 1947, and there have not been any improvements to the constant in the exponent for either bound since then.

However, there are graphs for which we can compute the Ramsey number exactly. Erd\H{o}s \cite{Erd47} showed that for a path $P_n$ on $n$ vertices, we have $R(P_n, K_m) = (n-1)(m-1) + 1$. The lower bound comes from considering the graph composed of $m-1$ disjoint red cliques of size $n-1$, with all edges between them blue. This lower bound construction was generalized by Burr \cite{Bur81}, who observed that for any connected graph $G$ and any graph $H$,
\begin{equation}\label{RamseyLowerBound}
R(G,H) \geq  (|G| - 1)(\chi(H) - 1) + \sigma(H).
\end{equation}
where $\chi(H)$ is the chromatic number of $H$ and $\sigma(H)$ is the size of the smallest color class in a $\chi(H)$-coloring of $H$. To see that \cref{RamseyLowerBound} holds, consider the graph composed of $\chi(H) - 1$ disjoint red cliques of size $|G|-1$ and one additional red clique of size $\sigma(H) - 1$, with all edges between the cliques blue. This graph has no red copy of $G$ because every red connected component has size at most $|G| - 1$, and it has no blue copy of $H$ because otherwise this copy would be partitioned, via the red cliques, into $\chi(H)$ parts with one part having size $\sigma(H) - 1$, contradicting the minimality of $\sigma(H)$.

We say that $G$ is $H$-good when equality holds in \cref{RamseyLowerBound}. The notion of Ramsey goodness was introduced by Burr and Erd\H{o}s \cite{BE} in 1983, and has been studied extensively since then, see e.g., \cite{ABS, CFLS, FGMSS, Nik, NR} and their references. Note that Erd\H{o}s' argument which gives a lower bound on $R(K_n, K_n)$ can be used to show that if we have relatively dense graphs $G,H$, then the Ramsey number is super-polynomial in $|G|$ and hence $G$ is not $H$-good. Thus we restrict our attention to sparse and connected $G$. In 1977, Chv\'{a}tal \cite{Ch} showed that any tree is $K_m$-good. Recently, Pokrovskiy and Sudakov \cite{PSpaths} showed that any path $P$ with $|P| \geq 4|H|$ is $H$-good, verifying a conjecture of Allen, Brightwell, and Skokan \cite{ABS} in a strong sense. 

Since paths are a special case of bounded degree trees, it is natural to consider whether trees are Ramsey good for all graphs $H$. 
In~\cite{EFRSTreeMultipartite} Erd\H{o}s, Faudree, Rousseau and Schelp ask ``What is the behavior of $R(T, K(n, n))$ when $T$ has bounded degree?''
Erd\H{o}s, Faudree, Rousseau and Schelp \cite{EFRSTreeMultipartite, EFRSMultipartiteSparse, EFRSMultipartiteTree} wrote several papers on this topic. The result in their 1985 paper \cite{EFRSMultipartiteSparse} implies that for any $H$, all sufficiently large bounded degree trees $T$ are $H$-good.
Though they do not give an explicit dependency between $|T|$ and $|H|$, their proof method can be used to show that  any bounded degree tree $T$ with $|T| \geq \Omega(|H|^4)$, is $H$-good. 
In this paper, we improve their result as follows.
 
\begin{theorem} \label{Tmain}
For all $\Delta$ and $k$ there exists a constant $C_{\Delta, k}$ such for any tree $T$ with max degree at most $\Delta$ and any $H$ with $\chi(H) = k$ satisfying $|T| \geq  C_{\Delta, k} |H| \log^4|H|$, $T$ is $H$-good.
\end{theorem}
The dependency between $|T|$ and $|H|$ in the above theorem is tight up to the  $\log|H|$ factors. 
Indeed  for $|T|\leq m=|K_m^k|/k$, no tree $T$ is $K_{m}^k$-good for the balanced complete multipartite graph $K_m^k$. To see this, consider, an edge colouring of a complete graph on $(2k-1)(|T|-1) + 1$ vertices 
consisting of $2k-1$ red cliques of size $|T|-1$, with all other edges blue. 
It is easy to check that this graph has no red $T$ and no blue  $K^k_m$ showing that $R(T, K_m^k)\geq (2k-1)(|T|-1) + 1>(k-1)(|T|-1)+m$.

In the proof of \Cref{Tmain}, we first consider the case where our tree $T$ has many leaves. In this case, we are able to obtain the following stronger result.

\begin{theorem} \label{Tmanyleaves}
Let $T$ be a tree with $l$ leaves and maximum degree at most $\Delta$, and let $H$ be a graph satisfying $l \geq 13 \Delta |H| + 1$. Then $T$ is $H$-good.
\end{theorem}

\begin{rem}
The condition $l \geq 13 \Delta |H| + 1$ can be replaced with $l \geq 13 \Delta m + 1$ where $m$ is the size of the largest color class in a $\chi(H)$ coloring of $H$. Indeed, this is what we actually prove in \Cref{L5}.
\end{rem}

\section{Overview}

\subsection*{Notation}

For a graph $G$, we let $E(G)$ denote the set of edges of $G$. We define $K^k_m$ to be the complete $k$-partite graph with each part having size $m$, where we let $K^1_m$ denote the empty graph on $m$ vertices. Also let $K_{m_1, \ldots, m_k}$ be the complete multipartite graph with parts of size $m_1, \ldots, m_k$. For a graph $G$ and vertex $x$, we let $N(x) = N_G(x) = \{ y \in G : xy \in E(G) \}$ denote the neighborhood of $x$. We analogously let $d_G(x) = |N_G(x)|$ denote the degree of $x$ and $\Delta(G)$ denote the maximum degree of a vertex in $G$. For any subset $S \subseteq G$, we define the neighborhood $N(S) = N_G(S) = \bigcup_{x \in S}{N_G(x)} \backslash S$.

\subsection*{Proof outline}

We are given a tree $T$ with $n$ vertices and a graph $H$ with $\chi(H) = k$ and $\sigma(H) = m_1$, and we would like to show that any graph $G$ on $(n-1)(k-1) + m_1$ vertices either has a copy of $T$, or $G^c$ has a copy of $H$. Note that as long as $k$ and $m_1$ are fixed, adding more edges to $H$ only makes the problem more difficult. Indeed, if we let $m_1 \leq \ldots \leq m_k$ be the sizes of the parts in a $k$-coloring of $H$, then a graph not containing $H$ also doesn't contain $K_{m_1, \ldots, m_k}$.
Because of this we will actually prove the following slightly stronger version of Theorem~\ref{Tmain}.
\begin{theorem} \label{TmainKmmmm}
For all $\Delta$ and $k$ there exists a constant $C_{\Delta, k}$ such for any tree $T$ with max degree at most $\Delta$ and numbers $m_1\leq  m_2\leq  \dots\leq  m_k$ with   $|T| \geq  C_{\Delta, k} m_k \log^4 m_k$, the tree $T$ is $K_{m_1, m_2, \dots, m_k}$-good.
\end{theorem}
Assume that we are given a graph $G$ on $(n-1)(k-1) + m_1$ vertices such that $G^c$ has no copy of $K_{m_1, \ldots, m_k}$. To prove Theorem~\ref{TmainKmmmm} we need to show that $G$ has a copy of $T$.
Notice that since  $G^c$ has no copy of $K_{m_1, \ldots, m_k}$, we have that  $G^c$ has no copy of $K^k_{m_k}$ and most of the time we will only use this weaker assumption.

A bare path in a graph is a path such that all interior vertices have degree 2. It is a well known result (see eg. Lemma~2.1 in~\cite{Krivelevich}) that a tree either has many leaves or many long bare paths. 
\begin{lemma} \label{L1}
For any integers $n, r > 2$, a tree on $n$ vertices either has at least $n/4r$ leaves or a collection of at least $n/4r$ vertex disjoint bare paths of length $r$ each.
\end{lemma} 

So we structure our paper into two parts. In \cref{Smanyleaves}, we suppose our tree $T$ has many leaves. We first describe the case $k=2$, i.e. so that $H = K_{m_1, m_2}$ is a complete bipartite graph. Then we observe, as in \cite{PSpaths, PScycles}, that a graph whose complement does not contain a complete bipartite graph has the property that large sets expand. After removing a small number of vertices, we obtain a graph which is an expander. We then make use of a theorem of Haxell \cite{Hax01} in order to embed the tree without leaves in our expander, and a generalization of Hall's theorem  to connect the leaves and complete the embedding. We then proceed by induction on $k$. In particular, we prove \Cref{Tmanyleaves} as a corollary. 

In \cref{Sfewleaves}, we consider instead the case where our tree has few leaves, and therefore many long bare paths by \Cref{L1}. In \cref{k2} we consider the case $k=2$, and again obtain an expander as above. We will often need to find disjoint paths of prescribed length between pairs of vertices, so we make the following definition.

\begin{definition}
For two sets $X$ and $W$ in a graph, we say that $(X,W)$ is $(s, d^-, d^+)$-linked system if the following holds.
Suppose that we have distinct vertices $x_1, y_1, \dots, x_s, y_s \in X$, and integers $d_1, \dots, d_s$ with $d^-\leq d_i\leq d^+$ for all $i$. Then there are disjoint paths $P_1, \dots, P_s$ with $P_i$ going from $x_i$ to $y_i$, $P_i$ internally contained in $W$, and $P_i$ having length $d_i$.
\end{definition}

We then follow the approach of Montgomery \cite{Mo14}, who shows that an expander is a $(s, d^-, d+)$-linked system for some appropriate choices of $s, d^-, d^+$ (\Cref{L2} and \Cref{T1}.) Thus we first apply Haxell's theorem in order to embed the tree with the paths removed and then apply Montgomery's result to find the required bare paths, completing the embedding. We finish \cref{k2} by combining the results for trees with many leaves and few leaves, thereby verifying \Cref{TmainKmmmm} for $k = 2$.

In \cref{k3} we consider the case $k \geq 3$. We first find $k-1$ disjoint subsets in $G$ so that they are sufficiently large and there are no edges between any 2 parts, and then make each of the parts an expander by removing a few vertices.  Next we look for sufficiently many short (length at most 3) paths between the $k-1$ parts and create an auxiliary graph on $[k-1]$ where there is an edge between $i$ and $j$ iff there are sufficiently many short paths between parts $i$ and $j$. 

If the auxiliary graph is nonempty, we take any nonempty connected component of it and consider the subgraph consisting of the parts of our original graph corresponding to that component, together with the short paths between them. Since each part is an expander and therefore a linked system by Montgomery \cite{Mo14} and there are many short paths connecting the linked systems, we can conclude that the whole subgraph is a linked system (\Cref{LemmaJoinManyLinkedSets}.) By also considering the neighborhoods of the parts, we are able to find a copy of our tree with the paths removed, as well as the forest of those paths. We then use the linked system in order to connect the required paths, completing the embedding (\Cref{LemmaLinkedKmkFreeTree}.)

Otherwise if the auxiliary graph is empty, then the neighborhoods of the $k-1$ parts in our original graph are sets that have no edges between them and have size at least $.9n$, so that our graph is close to the extremal construction. This case is dealt with separately in \Cref{L9}. By removing a few vertices from each set, we make each set an expander. Now if there is a vertex $v$ outside of the sets that has at least $\Delta$ neighbors to at least 2 sets, then because we can find a vertex that separates the tree into 2 forests with size at most $2n / 3$, we can apply a generalization of Haxell's theorem (\Cref{L3}) to find the 2 forests in those 2 sets with roots being exactly the neighbors of $v$, thus finding a copy of $T$.

Otherwise if all vertices outside the sets have at least $\Delta$ neighbors to at most 1 set, then we can place them in the set in which they have the most neighbors. This creates a partition of $G$ into $k-1$ parts with the property that no vertex in a part has more than $\Delta$ neighbors to any other part. Finally, we remove a few vertices from each part to make them expanders. If all the parts have at most $n-1$ vertices, then we must have removed at least $m_1$ vertices, and so we can take these vertices together with appropriate subsets of size $m_2, \ldots, m_k$ of the $k-1$ parts to get a copy of $K_{m_1, \ldots, m_k}$ in $G^c$, a contradiction. Hence there must be some part with at least $n$ vertices. Since this part is also an expander and has no copy of $K^2_{m_k}$, we can apply the result for $k=2$ to obtain a copy of $T$.

\section{Embedding a tree with many leaves} \label{Smanyleaves}

To deal with the case where where our tree has many leaves, we will need a result of Haxell \cite{Hax01}, which lets us embed a bounded degree tree with prescribed root into a graph with sufficient expansion. In \cref{k3}, we will actually need a generalization of this result to forests, so we state the more general version in the following lemma. For a proof of \Cref{L3}, we refer the reader to the appendix. 

\begin{lemma} \label{L3}
Let $\Delta, M, t$ and $m$ be given. Let $X = \{x_1, \ldots, x_t\}$ be a set of vertices in a graph $G$. Suppose that we have rooted trees $T_{x_1}, \ldots, T_{x_t}$ satisfying $\sum_{i=1}^t{\left|T_{x_i}\right|} \leq M$ and $\Delta \left(T_{x_i} \right) \leq \Delta$ for all $i$. Suppose that for all $S$ with $m \leq |S| \leq 2m$ we have $|N(S)| \geq M + 10 \Delta m$, and for $S$ with $|S| \leq m$ we  have $|N(S)\setminus X|\geq 4\Delta|S|$.

Then we can find disjoint copies of the trees $T_{x_1}, \ldots, T_{x_t}$ in $G$ such that for each $i$, $T_{x_i}$ is rooted at $x_i$. In addition for all $S \subseteq T_{x_1} \cup \ldots \cup T_{x_t}$ with $|S| \leq m$, we have
\begin{equation*}
\left| N(S)\setminus \big(T_{x_1}\cup\dots\cup T_{x_t}\big) \right| \geq \Delta|S|.
\end{equation*}
\end{lemma}

As a corollary, we can embed a large bounded degree tree into a graph whose complement does not contain $K_{m_1, m_2}$.

\begin{corollary} \label{C1}
Let $\Delta, m_1, m_2$ be integers, $T$ be a forest with $\Delta(T) \leq \Delta$, and $G$ a graph with  $|G| \geq |T| + 13\Delta m_1 + m_2$  such that $G^c$ does not contain $K_{m_1, m_2}$. Then $G$ contains a copy of $T$. Additionally, for all $S \subseteq T$ with $|S| \leq m_1$, we have 
\[ |N(S) \backslash T| \geq \Delta |S|. \]
\end{corollary}
\begin{proof}
Since every forest $F$ is a subgraph of some tree on $|F|$ vertices, without loss of generality we may suppose that $T$ is a tree.

Since $G^c$ does not contain $K_{m_1, m_2}$, we have that for any $S \subseteq G$ with $m_1 \leq |S| \leq 2 m_1$, $|N_G(S)| \geq |G| - 2 m_1 - m_2$. Now if we choose $|X| \leq m_1-1$ maximal so that $|N_G(X)| \leq 4\Delta|X|$, then we claim that $G' = G \backslash X$ satisfies that for all $S \subseteq G'$ with $1 \leq |S| \leq m_1$, $|N_{G'}(S)| \geq 4 \Delta |S| + 1$. Indeed, for any $S \subseteq G'$ with $1 \leq |S| \leq m_1$, if $|N_{G'}(S)| \leq 4 \Delta |S|$ then $|N_G(X \cup S)|\leq |N_G(X)\cup N_G(S)| \leq 4 \Delta |X \cup S|$, so we must have $m_1 \leq |X \cup S| \leq 2 m_1$ by maximality of $X$. But then
\[ 8 \Delta m_1 \geq 4 \Delta |X \cup S| \geq |N_G(X \cup S)| \geq |G|- 2 m_1 - m_2,\]
contradicting the assumption of the lemma. For any $S \subset G'$ with $m_1 \leq |S| \leq 2 m_1$  we have
\[ |N_{G'}(S)| \geq |N_G(S)|-|X|\geq |N_G(S)| - m_1 \geq |G| - 3 m_1 - m_2 \geq |T| + 10 \Delta m_1.\]
Thus we may apply \Cref{L3} with the graph $G'$, $m=m_1$,  $X=\{x\}$ for any vertex $x$, and the tree $T_x=T$, to obtain that $G'$ contains a copy of $T$. Moreover, for all $S \subseteq T$ with $|S| \leq m_1$, we have
\[ |N_G(S) \backslash T| \geq |N_{G'}(S) \backslash T| \geq \Delta |S|.\]
\end{proof}

We will also need the following extension of Hall's theorem.

\begin{lemma} \label{L4}
Given a bipartite graph $(A,B)$ and a function $l : A \rightarrow \N$, if $|N(S)| \geq \sum_{v \in S}{l(v)}$ for all $S \subseteq A$ then the graph contains a forest $F$ such that $d_F(v) = l(v)$ for all $v \in A$ and $d_F(v) \leq 1$ for all $v \in B$.
\end{lemma}

We are now ready to prove that a bounded degree tree with sufficiently many leaves is $K_{m_1, \ldots, m_k}$-good.

\begin{lemma} \label{L5}
Let $l, \Delta, k \in \N$ and $m_1 \leq \ldots \leq m_k$ be given with $l \geq 13 \Delta m_k + 1$. Then any tree $T$ with $l$ leaves and $\Delta(T) \leq \Delta$ is $K_{m_1,\ldots, m_k}$ good.
\end{lemma}
\begin{proof}
Let $n = |T|$. We proceed by induction on $k$. For $k=1$, any graph on $m_1$ vertices trivially contains $K^{1}_{m_1}$ as a subgraph (since $K_m^1$ is the graph with $m$ vertices and no edges.) Now suppose $k \geq 2$ and let $G$ be a graph with $(k-1)(n-1) + m_1$ vertices such that $G^c$ does not contain $K_{m_1, \ldots, m_k}$. 

First suppose there exists $S \subseteq G$ with $|S| \geq m_k$, such that $|N_G(S)| \leq n - |S| - 1$. Then $|N_{G^c}(S)| \geq (k-2)(n-1) + m_1$ and $N_{G^c}(S)$ does not contain a $K_{m_1, \ldots, m_{k-1}}$, or else we could take it together with an $m_k$ vertex subset of $S$ to get a $K_{m_1, \ldots, m_k}$ in $G^c$. Thus we may apply induction to $N_{G^c}(S)$ to conclude that it contains a copy of $T$.

Otherwise, we have that for all $S \subseteq G$ with $|S| \geq m_k$, $|N_G(S)| \geq n - |S|$. For sets $S$ with $|S|=m_k$, this is equivalent to $G^c$ not containing $K_{m_k, m'}$ for  $m' = (k-2)(n-1) + m_1$.  Now let $T'$ be the subtree of $T$ with all leaves removed and fix $x_1 \in G$ to be any vertex. Using   $l \geq 13 \Delta m_k + 1$ we have
\[ (k-1)(n-1) + m_1 \geq n - l + 13 \Delta m_k + (k-2)(n-1) + m_1= n-l+13\Delta m_k +m'.\]
Combining this with $|T'| = n-l$, 
we can apply \Cref{C1} to conclude that $G$ contains a copy of $T'$ rooted at $x_1$. Now let $P$ be the vertices of $T'$ to which we need to connect leaves in order to get $T$, and let $l(v)$ be the number of leaves to attach for each $v \in P$. From the last part of \Cref{C1}, we have that for any $S \subseteq P$ with $|S| \leq m_k$, 
\[ |N_G(S) \backslash T'| \geq \Delta |S| \geq \sum_{v \in S}{l(v)}.\]
Moreover, for any $S \subseteq P$ with $|S| \geq m_k$, we have $|N_G(S)| \geq n - |S|$ which implies
\[ |N_G(S) \backslash T'| \geq |N_G(S)| - |T' \backslash S| = |N_G(S)| + |S| - n + l \geq l = \sum_{v \in P}{l(v)} \geq \sum_{v \in S}{l(v)}.\]
Thus we may apply \Cref{L4} to complete the embedding of $T$.
\end{proof} 

\Cref{Tmanyleaves} now follows immediately.

\begin{proof}[Proof of \Cref{Tmanyleaves}]
Let $n = |T|$, $k = \chi(H)$ and $m_1 \leq \ldots \leq m_k$ be the sizes of the color classes in a $k$-coloring of $H$, so that $m_1 = \sigma(H)$. Let $G$ be a graph on $(n - 1)(k - 1) + m_1$ vertices such that $G^c$ has no copy of $H$. Then $G^c$ has no copy of $K_{m_1, \ldots, m_k}$, and we have that $\ell \geq 13 \Delta |H| + 1 \geq 13 \Delta m_{k} + 1$, so by \Cref{L5} $G$ must contain a copy of $T$.
\end{proof}

\section{Embedding a tree with few leaves} \label{Sfewleaves}

If a bounded degree tree doesn't have many leaves, then it has many long bare paths by \Cref{L1}, so it remains to embed such trees. We will need the following definitions and lemmas of Montgomery \cite{Mo14}. First we define a notion of expansion into a subset of a graph.

\begin{definition}
For a graph $G$ and a set $W \subseteq G$, we say $G$ $d$-expands into $W$ if
\begin{enumerate}
\item $|N(X) \cap W| \geq d|X|$ for all $X \subseteq G$ with $1 \leq |X| < \left \lceil \frac{|W|}{2d} \right \rceil$.
\item $e(X,Y) > 0$ for all disjoint $X,Y \subseteq G$ with $|X| = |Y| = \left \lceil \frac{|W|}{2d} \right \rceil$.
\end{enumerate}
\end{definition}

\begin{definition}
We call $G$ an $(n,d)$-expander if $|G| = n$ and it $d$-expands into $G$.
\end{definition} 

We state some basic properties of expansion.

\begin{lemma} \label{subsetexpand}
Let $W \subseteq Z \subseteq G$ and suppose that $G$ $d$-expands into $W$.
\begin{enumerate}
\item[(i)] $Z$ $d$-expands into $W$.
\item[(ii)] If $d \geq 2$ then $G$ $d$-expands into $Z$.
\item[(iii)] If $d > 1$ and $d/(d-1) \leq c \leq d$ then  $G$ $c$-expands into $W$.
\end{enumerate}
\end{lemma}
\begin{proof}
$(i)$ follows directly from the definition. 

For $(ii)$, condition 2 follows immediately. For condition 1, let $X \subseteq G$ with $1 \leq |X| < \left \lceil \frac{|Z|}{2d} \right \rceil$ be given. If $|X| < \left \lceil \frac{|W|}{2d} \right \rceil$ then we have $|N(X) \cap Z| \geq |N(X) \cap W| \geq d|X|$. Otherwise if $\left \lceil \frac{|W|}{2d} \right \rceil \leq |X| < \left \lceil \frac{|Z|}{2d} \right \rceil$ then we know that $|Z \backslash ( N(X) \cup X) | < \left \lceil \frac{|W|}{2d} \right \rceil$ by condition 2 of $d$-expansion. It follows that 
\[ |N(X) \cap Z| \geq |Z| - |X| - \frac{|W|}{2d} \geq |Z| - \frac{|Z|}{2d} - \frac{|W|}{2d} \geq \frac{|Z|}{2} \geq d|X|. \]

The proof of $(iii)$ is similar to that of $(ii)$. The interesting case to check is when $\left \lceil \frac{|W|}{2d} \right \rceil \leq |X| < \left \lceil \frac{|W|}{2c} \right \rceil$, which implies $| W \backslash (N(X) \cup X) | < \left \lceil \frac{|W|}{2d} \right \rceil$ by condition 2 of $d$-expansion. Notice that $d/(d-1) \leq c$ is equivalent to $c^{-1}+d^{-1}\leq 1$. Combining these gives
\[  |N(X) \cap W| \geq |W| - |X| - \frac{|W|}{2d} \geq |W| - \frac{|W|}{2c} - \frac{|W|}{2d} \geq \frac{|W|}{2} \geq c |X|. \]
\end{proof}

We will also need a useful decomposing property of this expansion.

\begin{lemma}[Lemma 2.3 of Montgomery \cite{Mo14}] \label{annoying}
There exists $n_0$ such that for $k,n \in \N$ with $n \geq n_0$ and $k \leq \log{n}$, if we have $m_1, \ldots, m_k \in \N$ with $m = m_1 + \ldots + m_k$ and $d_i = \frac{m_i}{5m}d \geq 2 \log{n}$, then for any graph $G$ with $n$ vertices which $d$-expands into $W$ with $|W| = m$, we can partition $W$ into $k$ disjoint sets $W_1, \ldots, W_k$ of sizes $m_1, \ldots, m_k$ respectively, so that $G$ $d_i$-expands into $W_i$.
\end{lemma}

The following lemma will be crucial for \cref{k3}. It allows us to simultaneously find many paths of prescribed lengths between endpoints in an expander graph.

\begin{lemma}[Lemma 3.2 of Montgomery \cite{Mo14}] \label{L2}
Let $G$ be a graph with $n$ vertices, where $n$ is sufficiently large, and let $d = 160 \log{n} / \log{\log{n}}$. Suppose $r, k_1, \ldots, k_r$ are integers with $4\lceil \log{n} / \log{\log{n}}\rceil \leq k_i \leq n/40$, for each $i$, and $\sum_{i}{k_i} \leq 3|W| / 4$. Suppose $G$ contains disjoint vertex pairs $(x_i, y_i), 1 \leq i \leq r$, and let $W \subset G$ be disjoint from these vertex pairs.

If $G$ $d$-expands into $W$, then we can find disjoint paths $P_i$, $1 \leq i \leq r$,with interior vertices in $W$, so that each path $P_i$ is an $x_i,y_i$-path with length $k_i$.
\end{lemma}

It will be convenient for us to restate the previous lemma using the definition of a linked system.

\begin{corollary}\label{LemmaExpanderLinked}
Let $n, s \in\N$ and $c = 160\log n/\log\log n$. Suppose that $G$ is a graph on $n$ vertices and $W \subseteq G$ such that $n \geq |W| + 2s$ and $G$ $c$-expands  into $W$. Then $(G \backslash W, W)$ is $(s,d^-, d^+)$-linked system, for $d^-= 4\left\lceil \frac{\log n}{\log\log n}\right\rceil$ and $d^+ = \frac{|W|}{40s}$.
\end{corollary}
\begin{proof}
This follows immediately from \Cref{L2} and the definition of $(s, d^-, d^+)$-linked system.
\end{proof}

Lemma~\ref{L2} shows that if a graph $G$ expands into a set $W$, then it is possible to cover $3/4$ of $W$ by disjoint paths of prescribed length. The following theorem shows that, under similar assumptions to Lemma~\ref{L2}, it is possible to cover all of $W$ by such paths.
%\Cref{L2} is also crucial for Montgomery, who uses it together with absorbers to deduce the following.

\begin{theorem}[Theorem 4.3 of Montgomery \cite{Mo14}] \label{T1}
Let $n$ be sufficiently large and let $l \in \N$ satisfy $l \geq 10^3 \log^2{n}$ and $l | n$. Let a graph $G$ contain $n/l$ disjoint vertex pairs $(x_i, y_i)$ and let $W = G \backslash (\cup_i\{x_i, y_i\})$. Suppose $G$ $d$-expands into $W$, where $d = 10^{10} \log^4{n} / \log{\log{n}}$. Then we can cover $G$ with $n/l$ disjoint paths $P_i$ of length $l-1$, so that $P_i$ is an $x_i,y_i$-path.
\end{theorem}

Montgomery uses the above theorem to embed a spanning tree with many long bare paths in an expander. The idea is to first find a copy of the tree with the bare paths removed, and then apply \Cref{T1} to find the paths. We will use this theorem for the same purpose in \cref{k2}.

\subsection{The case $k=2$} \label{k2}

If we have a graph with at least $n$ vertices for which small sets expand and whose complement does not contain $K^2_m$, then we can find an embedding of the tree via \Cref{T1}, as in Montgomery \cite{Mo14}.

\begin{lemma} \label{L6}
Let $n, m, \Delta \in \N$ with $n$ sufficiently large relative to $\Delta$ and let $d = 4 \cdot  10^{12} \frac{\log^4{n}}{\log{\log{n}}}$, $r = \lceil 10^3 \log^2{n} \rceil$, such that $n \geq 2(d+1)m$. Let $T$ be a tree with $n$ vertices, $\Delta(T) \leq \Delta$, and at least $n/(4r)$ disjoint bare paths of length $r$. If $G$ is a graph with $n'$ vertices such that $n' \geq n$, $G^c$ does not contain $K^2_m$, and for all $S \subset G$ with $|S| \leq m$, $|N(S)| \geq d|S|$, then $G$ contains a copy of $T$.
\end{lemma}
\begin{proof}
If $n' \geq n + 13 \Delta m + m$ then $G$ contains a copy of $T$ by \Cref{C1}. Otherwise we have $n \leq n' < n + 13 \Delta m + m = n(1 + o(1))$. 

We first note that $G$ is an $(n', d)$-expander. Indeed, for any $S \subseteq G$ with $1\leq |S|\leq m$ we have $|N(S)| \geq d|S|$ by assumption. For  $S \subseteq G$ with $m \leq |S| < \lceil n'/(2d) \rceil$, using $n'\geq 2(d+1)m$  and the $K_m^2$-freeness of $G^c$  we have
\[ |N(S)| \geq n' - |S| - m \geq d|S|, \]
so the first condition holds. Moreover, since $G^c$ does not have $K^2_m$ and $\lceil n'/(2d) \rceil \geq m$, the second condition holds as well.

Now let $T'$ be $T$ with the interior vertices of the $n/4r$ bare paths of length $r$ deleted. Then $|T'| = 3n/4 + n/(4r)$. Let $n_1 = n'-n/8$ and $n_2 = n / 8$. Then if we let $d_i = \frac{n_i}{5n'} d$, we can apply \Cref{annoying} to partition $G$ into $G_1$ and $G_2$ such that $|G_i| = n_i$ and $G$ $d_i$-expands into $G_i$. Note that $m = o(n)$ and hence
\[ n_1 = n'-n/8 \geq \frac{7n}{8} \geq \frac{3n}{4} + \frac{n}{4r} + 13 \Delta m + m = |T'| + 13 \Delta m + m.\]
Morever, $G_1^c$ has no $K^2_m$ so we conclude by \Cref{C1} that $G_1$ contains a copy of $T'$. Let $(x_i, y_i)$ be the disjoint vertex pairs in the copy of $T'$ that need to be connected by paths to get $T$.

Let $G'$ be any subgraph of $G$ of size $(r+1)n/4r$ containing $G_2 \cup (\bigcup_{i}{\{x_i, y_i\}})$, and let $W=G'\setminus (\bigcup_{i}{\{x_i, y_i\}})$.
Since $G_2\subseteq W$,  we may apply \Cref{subsetexpand} $(i), (ii)$ to conclude that $G'$ $d_2$-expands into $W$. 
We have 
$$d_2 = dn/40n' \geq d/41\geq 10^{10} \log^4{n} / \log{\log{n}} \geq 10^{10} \log^4{|G'|} / \log{\log{|G'|}}.$$
By \Cref{subsetexpand} (iii), $G'$ $10^{10} \log^4{|G'|} / \log{\log{|G'|}}$-expands into $W$.
Since $|G'|\leq n$, we have and $r+1 \geq 10^3 \log^2{|G'|}$.
Combining these, we can apply \Cref{T1} with $l = r+1$, $G=G'$, and $d=10^{10} \log^4{|G'|} / \log{\log{|G'|}}$ to conclude that the pairs $(x_i, y_i)$ can be connected by disjoint paths of length $r$ in $G'$, completing the embedding of $T$.
\end{proof}

Putting \Cref{L5} and \Cref{L6} together, we may conclude the case $k=2$ for all bounded degree trees as follows.

\begin{proof}[Proof of \Cref{TmainKmmmm} for $k=2$]
Let $d = 4 \cdot 10^{12} \frac{\log^4{n}}{\log{\log{n}}}$ and $r = \lceil 10^3 \log^2{n} \rceil$. We can choose $C_{\Delta,k}$ such that $n$ is sufficiently large relative to $\Delta, k$ and $n \geq (2d+3)m_2$. Now let $G$ be a graph with $n+m_1-1$ vertices such that $G^c$ does not contain $K_{m_1, m_2}$. Notice that in particular, $G^c$ doesn't contain $K_{m_2}^2$.
If $T$ has at least $n/4r \geq 13 \Delta m_2 + 1$ leaves, then by \Cref{L5} we are done. Otherwise, by \Cref{L1} $T$ has at least $n/4r$ disjoint bare paths of length $r$. Note that since $G^c$ has no $K_{m_1, m_2}$, we have that for any $S \subseteq G$ with $|S| \geq m_1$, $|N(S)| \geq n - m_2-|S|$. Now choose $X \subseteq G$ with $|X| \leq m_1-1$ maximal so that $|N(X)| < d|X|$ and let $G' = G \backslash X$. Then we claim that for all $S \subseteq G'$ with $|S| \leq m_1$, $|N_{G'}(S)| \geq d|S|$. Indeed, otherwise we would have $|N(X \cup S)| < d|X \cup S|$, so by maximality of $X$ this would imply $m_1 \leq |X \cup S| \leq 2m_1$. But then
\[ 2dm_1 \geq d | X \cup S | > |N(X \cup S)| \geq n - m_2-|X \cup S| \geq n - 2m_1-m_2,\]
a contradiction. 
For $S$ with $m_1\leq |S|\leq m_2$ we have $|N(S)|\geq n-2m_2\geq dm_2\geq d|S|$. Since $|X|\leq m_1-1$, we have $|G'|\geq n$.
Thus we may apply \Cref{L6} to $G'$ with $m=m_2$ to conclude that $G'$ contains a copy of $T$.
\end{proof}

\subsection{The case $k \geq 3$} \label{k3}

We first extend \Cref{C1} to show that we can embed a large bounded degree tree into a graph whose complement does not contain $K^k_m$.
 
\begin{lemma} \label{C2}
Let $\Delta, k, m\in \N$ be given,   $T$ a tree with $\Delta(T) \leq \Delta$, and $G$ a graph with  $|G| \geq (k-1)(|T| + 13 \Delta m) + m$  such that $G^c$ does not contain $K^k_m$. Then $G$ contains a copy of $T$.
\end{lemma}
\begin{proof}
We proceed by induction on $k$. For $k=1$, any graph on $m$ vertices trivially contains $K^1_m$. Now suppose $k \geq 2$ and let $m' = (k-2)(|T| + 13 \Delta m) + m$. If $G^c$ does not contain $K_{m, m'}$ then by \Cref{C1}, $G$ contains a copy of $T$.

Otherwise $G$ contains disjoint $A,B$ with $|A| = m, |B| = m'$ and no edges between $A$ and $B$. Then $B^c$ does not contain a copy of $K^{k-1}_m$ or else taking this copy together with $A$ would give a copy of $K^k_m$ in $G^c$. But then by induction, $B$ contains a copy of $T$.
\end{proof}

Moreover, for $k \geq 3$, we observe that we can embed much larger bounded degree forests than trees. This makes sense in view of the Burr's construction showing (\ref{RamseyLowerBound}) -- it does not have a tree on $n$ vertices, but it has a forest made of $k-1$ trees each of size $n-1$. 

\begin{corollary} \label{C3}
Let $k, m, \Delta \in \N$ be given with $k \geq 3$,  and let $T_a, T_b$ be trees with $|T_a|\leq |T_b|$ and  $\Delta(T_a), \Delta(T_b) \leq \Delta$. Let $G$ be a graph with $G^c$ not containing $K^k_m$. If 
\[ |G| \geq |T_a| + (k-1)(|T_b| + 13 \Delta m) + m,\]
then $G$ contains a copy of the forest $T_a \cup T_b$.
\end{corollary}
\begin{proof}
We first apply \Cref{C2} to obtain a copy of $T_a$ in $G$. Now we let $G' = G \backslash T_a$ and apply \Cref{C2} to $G'$ to obtain a copy of $T_b$ in $G'$.
\end{proof}

The following lemma lets us find a copy of $T$ in a sufficiently large graph which contains a linked system and whose complement is $K^k_m$-free, but does contain $K^{k-1}_u$, for a sufficiently large $u$. 
The idea of the proof is to break up our tree into three parts---two forests $T_a$, $T_b$, and a collection of bare paths joining the forests. Then the forests $T_a$ and $T_b$ are found using Corollary~\ref{C3}, while the bare paths are found using the linked system.

\begin{lemma}\label{LemmaLinkedKmkFreeTree}
Let $n, m, k, \Delta \in \N$ with $k \geq 3$ and $n$ sufficiently large relative to $\Delta, k$ and let $d = 4 \cdot 10^{12} \frac{\log^4{n}}{\log{\log{n}}}$, $r = \lceil 10^3 \log^2{n}\rceil$ and $y = \lceil \log{n} \rceil$, such that $n\geq 2(d+1) m$.
Let $X, W, Z$ be disjoint subsets of a graph such that $(Z\cup X)^c$ is $K_m^k$-free with $|Z|\geq 0.99(k-1)n$.  
Let $T$ be a tree on $n$ vertices with $\Delta (T)\leq \Delta$, and at least $n/4r$ bare paths of length $r$.
Suppose that $X^c$ contains $K_u^{k-1}$ for $u\geq 2n/r$.
Suppose that $(X,W)$ is a $(n/2r, d^-, d^+)$-linked system for $d^-\leq y\leq d^+$.
Then $Z\cup X\cup W$ contains a copy of $T$.
\end{lemma}
\begin{proof}
We first find a subset of $Z$ with appropriate expansion properties.
\begin{claim}\label{ClaimZExpands}
There exists $Z' \subseteq Z$ with $|Z'| \geq 0.9(k-1)n$ such that $|N(S)\cap X|\geq |S|$ for any $S \subseteq Z'$ with $|S|\leq n/r$.% with $|S|\leq m$ and $|N(S)\cap X|\geq n/r$ for any $S\subseteq Z'$ with $|S|\geq m$.
\end{claim}
\begin{proof}
Let $U_1, \dots, U_{k-1}$ be the parts of the $K_u^{k-1}$ in $X^c$.
If there exists $S \subseteq Z$ with $|S| \geq m$ and $|N_{G^c}(S) \cap U_i| \geq m$ for all $i$, then we can take subsets of size $m$ from $S, N_{G^c}(S) \cap U_1,  \ldots, N_{G^c}(S) \cap U_{k-1}$ to obtain a $K^k_m$ in $\left( X \cup Z \right)^c$, a contradiction. Thus for all $S \subseteq Z$ with $n/r\geq |S| \geq m$, we have that $|N_G(S) \cap X| \geq u - m\geq n/r\geq |S|$ (using $m=o(n)$.)

Now let $A \subset Z$ with $|A| \leq m-1$ be maximal such that $|N_G(A) \cap X| <   |A|$, and let $Z' = Z \backslash A$. We claim that for all $S \subseteq Z'$ with $|S| \leq m$, $|N_G(S) \cap X| \geq |S|$. Indeed, otherwise $|N_G(A \cup S) \cap X| < |A \cup S|$, so we must have $m \leq |A \cup S| \leq 2m$ by maximality of $A$. But then
\[ 2m \geq |A \cup S| > |N_G(A \cup S)| \geq n/r,\]
a contradiction to $n\geq 2(d+1)m$. 
\end{proof}

Now let $T_a$ be a collection of $n/4r$ disjoint paths of length $r - 2y - 4$, so that $|T_a| = n(r - 2y - 3)/(4r) \leq n/4$ and let $T_b$ be $T$ without the interior vertices of the $n/4r$ bare paths of length $r$, so that $|T_b| = n - n(r-1)/4r = 3n/4 + n/(4r)$. Since we can always add edges to $T_a$ and $T_b$ to make them trees without increasing the maximum degree, and 
\[ |Z'| \geq 0.9(k-1)n \geq \frac{n}{4} + (k-1)\left(\frac{3n}{4} + \frac{n}{4r} + 13 \Delta m\right) + m, \]
we may apply \Cref{C3} to conclude that $Z'$ has a copy of $T_a$ and $T_b$. Let $x_a, y_a \in Z'$ for $1 \leq a \leq n/2r$ be the endpoints of those copies so that if we connect $x_a$ with $y_a$ by disjoint paths of length $y + 2$ for all $i$, we obtain an embedding of $T$. 
By Lemma~\ref{L4} and the claim, there is a matching from $\{x_a: 1\leq a\leq n/2r\}\cup \{y_a: 1\leq a\leq n/2r\}$ to some set $\{x'_a: 1\leq a\leq n/2r\}\cup \{y'_a: 1\leq a\leq n/2r\}$ contained in $X$.  
Since $(X,W)$ is a $(n/2r, d^-, d^+)$-linked system for $d^-\leq y\leq d^+$, there are disjoint $x'_a$ to $y'_a$ paths of length $y$ in $W$ as required.
\end{proof}

Next we prove two lemmas which help us construct linked systems. \Cref{LemmaJoinTwoLinkedSets} lets us combine 2 linked systems into a bigger linked system, provided that there are sufficiently many short paths between them. In \Cref{LemmaJoinManyLinkedSets}, we combine several linked systems with many short paths between them into a big linked system, by making repeated use of \Cref{LemmaJoinTwoLinkedSets}.

\begin{lemma}\label{LemmaJoinTwoLinkedSets}
Suppose that we have sets of vertices $X_1, X_2, W_1, W_2$ with $(X_1\cup W_1)\cap (X_2\cup W_2)=\emptyset$, such that $(X_1,W_1)$ is a $(s_1, d_1^-, d_1^+)$-linked system  and $(X_2, W_2)$ is a $(s_2, d_2^-, d_2^+)$-linked system.
Suppose that there are disjoint paths $P_1, \dots, P_t$ of length $\leq 3$ from $X_1$ to $X_2$ internally outside $X_1\cup X_2\cup W_1\cup W_2$. 
Then $\left(X_1\cup X_2, W_1\cup W_2\cup\bigcup_{i=1}^t P_t\right)$ is a $(s, d^-, d^+)$-linked system  for $d^-=d^-_1+d^-_2 + 3$, $d^+=\min(d^+_1, d^+_2)$, and $s = \min(s_1, s_2, t/3).$
\end{lemma}
\begin{proof}
Let $x_1, y_1,  \dots, x_s, y_s$ be vertices in $X_1\cup X_2$  and $d_1, \dots d_s\in [d^-, d^+]$ as in the definition of $(s, d^-, d^+)$-linked system.  To prove the lemma we need to find disjoint paths $Q_1, \dots, Q_s$ with $Q_i$ a length $d_i$ path from $x_i$ to $y_i$.
Without loss of generality $x_1, y_1,  \dots, x_s, y_s$ are labeled so that $x_1, y_1, \dots, x_a, y_a\in X_1$, $x_{a+1}, y_{a+1}, \dots, x_b, y_b\in X_2$, $x_{b+1}, \dots, x_{s}\in X_1$, and $y_{b+1}, \dots, y_{s}\in X_2$ for some $a$ and $b$. 

Since the paths $P_1, \dots, P_t$ are disjoint and have only $2$ vertices each in $X_1\cup X_2$, we have that  $\leq 2s$ of the paths $P_1, \dots, P_t$ intersect $\{x_1, \dots, x_s, y_1, \dots, y_s\}$.
Since $t\geq 3s$, without loss of generality, we can suppose that the paths $P_{b+1}, \dots, P_{s}$ are disjoint from $\{x_1, \dots, x_s, y_1, \dots, y_s\}$. For each $i=b+1, \dots, s$, let $y'_i$ be the endpoint of $P_{i}$ in $X_1$, and $x'_i$ the endpoint of $P_i$ in $X_2$.
For each $i=b+1, \dots, s$, let $d^1_i=d_1^-$ and $d^2_i=d_i-d_1^- - |E(P_i)|$.
 Notice that by assumption we have $d^-_1+d^-_2 + 3=d^-\leq d_i\leq d^+=\min(d^+_1, d^+_2)$ which combined with $|E(P_i)| \leq 3$ implies that $d^-_1\leq d^1_i\leq d^+_1$ and   $d^-_2\leq d^2_i\leq d^+_2$.
 
 Apply the definition of $(X_1, W_1)$ being a $(s_1, d_1^-, d_1^+)$-linked system  in order to find paths $Q_1, \dots, Q_a, Q_{b+1}^1, \dots, Q_{s}^1$ with $Q_i$ a length $d_i$ path  from $x_i$ to $y_i$ internally contained in $W_1$, and $Q_i^1$ a length $d_i^1$ path  from $x_i$ to $y_i'$ internally contained in $W_1$.
 Similarly, apply the definition of $(X_2, W_2)$ being a $(s_2, d_2^-, d_2^+)$-linked system  to find paths $Q_{a+1}, \dots, Q_b, Q_{b+1}^2, \dots, Q_{s}^2$ with $Q_i$ a length $d_i$ path  from $x_i$ to $y_i$ internally contained in $W_2$, and $Q_i^2$ a length $d_i^2$ path  from $x_i'$ to $y_i$ internally contained in $W_2$.
For $i= b+1, \dots, s$, let $Q_i=Q_{i}^1+P_i+Q_{i}^2$ to get a length $d_i=d_i^1+d_i^2 + |E(P_i)|$ path going from $x_i$ to $y_i$.
 Now the paths $Q_1, \dots, Q_s$ are paths from $x_1, \dots, x_s$  to $y_1, \dots, y_s$ internally contained in $W_1\cup W_2\cup\bigcup_{i=1}^t P_t$  as in the definition of $(s, d^-, d^+)$-linked system.
\end{proof}

\begin{lemma}\label{LemmaJoinManyLinkedSets}
Let $G$ be a graph and $k, s, d^-, d^+\in \N$. For $i=1, \dots, k$ suppose that we have a $(s, d^-, d^+)$-linked system  $(X_i, W_i)$ with $(X_i\cup W_i)\cap (X_j\cup W_j)=\emptyset$ for $i\neq j$.
Suppose that we have a connected graph $F$ with vertex set $\{1, \dots, k\}$ such that for all $uv\in E(F)$ there is a family $\mathcal P_{uv}$ of $t$ disjoint paths of length $\leq 3$ from $X_u$ to $X_v$ internally outside $\bigcup_{i=1}^k X_i\cup W_i$ with $t\geq 15 k s$.
Then    $(X, W)$ is a $( s,  k(d^- + 3),  d^+)$-linked  system  for $X=X_1\cup\dots\cup X_k$ and $W=W_1\cup\dots\cup W_k\cup\bigcup_{e\in H} \mathcal P_{e}$.
\end{lemma}
\begin{proof}
Without loss of generality we can suppose that $F$ is a tree with edges $e_2, \dots, e_{k}$, and that the vertices of $F$ are ordered so that for each $i$, the edge $e_i$ goes from vertex $i$ to some vertex in $\{1, \dots, i-1\}$. Notice that this ensures that the induced subgraph $F[\{1, \dots, i\}]$ is a tree for every $i$.

For all $e_i\in E(F)$, choose a subfamily $\mathcal P'_{e_i}\subseteq \mathcal P_{e_i}$ with $|\mathcal P'_{e_i}|=3s$ 
 such that the paths in $\mathcal P'_{e_i}$ are disjoint from those in $\mathcal P'_{e_j}$ for $i\neq j$. 
This is done by choosing the paths in $\mathcal P'_{e_i}$ one by one for each $i$ always choosing them to be disjoint from  $\bigcup_{j=2}^{i-1}\bigcup_{P\in \mathcal P'_{e_j}}{P}$.
This is possible since $|\bigcup_{j=2}^{i-1}\bigcup_{P\in \mathcal P'_{e_j}}{P}| \leq 12is$ (using the fact that the paths in all $\mathcal P_{e_j}$ have length $\leq 3$), and since there are $t\geq 15ks>12is+3s$ paths in $P_{e_{i}}$ which are all disjoint.

We will use induction on $i$ to prove that ``($X_1\cup\dots\cup X_i$,  $W_1\cup\dots\cup W_i\cup\bigcup_{j= 2}^i \mathcal P'_{e_j}$) is a $( s,  i(d^- + 3),  d^+)$-linked system.'' The initial case ``$i=1$''  follows from  $(X_1, W_1)$ being a $(s, d^-, d^+)$-linked system.
Suppose that $i\geq 2$, and $(X', W')$ is a $( s,  (i-1)(d^- + 3),  d^+)$-linked system for  $X'=X_1\cup\dots\cup X_{i-1}$  and $W'=W_1\cup\dots\cup W_{i-1}\cup\bigcup_{j= 2}^{i-1} \mathcal P'_{e_j}$.

By construction of $\mathcal P'_{e_{i}}$ and the initial assumption that paths in $\mathcal P_{e_i}$ are internally disjoint from $\bigcup_{j=1}^k X_j\cup W_j$ we have that paths in $\mathcal P'_{e_i}$ are internally disjoint from $X'\cup W'$ and $X_i\cup W_i$.
From the lemma's assumptions,  for $a< b$ we have $(X_a\cup W_a)\cap (X_b\cup W_b)=\emptyset$ and we know that paths in  $\mathcal P_{e_a}$ are disjoint from $X_b\cup W_b$. These imply $(X'\cup W')\cap (X_i\cup W_i)=\emptyset$.
 Also, since $e_i\in E(F)$, we have that every path in $\mathcal P'_{e_i}$ goes from $X'$ to $X_i$ and has length $\leq 3.$
By Lemma~\ref{LemmaJoinTwoLinkedSets}, we have that $(X'\cup X_i, W'\cup W_i\cup\bigcup_{j= 1}^i \mathcal P'_{e_j})$ is a $( \min(s, |\mathcal P'_{e_i}|/3) ,  (i-1)(d^- + 3)+d^-+3,  d^+)$-linked system. Since $|\mathcal P'_{e_i}|/3=s$, this completes the induction step.
\end{proof}

We will need the well known folklore result that every tree  $T$ can be separated into two parts of size $\leq 2|T|/3$ with one vertex (see e.g. \cite{Chung}, Corollary 2.1.)
\begin{lemma} \label{L8}
The vertices of any tree $T$ can be partitioned into a vertex  $u$ and two disjoint sets $T_a$ and $T_b$ such that $|T_a|, |T_b|\leq 2n/3$ and there are no edges between $T_a$ and $T_b$.
\end{lemma}

%A balanced separator of a graph $G$ is a subset $X \subseteq G$ such that $G \backslash X$ can be partitioned into 2 parts of size at most $2 |G| / 3$ with no edges between them. The following lemma is folklore, and so we state it without proof. 

%\begin{lemma} \label{L8}
%Any tree has a one vertex balanced separator.
%\end{lemma}

The following lemma shows that if we have a $2$-edge-coloured complete graph on $(k-1)(n-1) + m_1$ vertices  whose colouring is close to Burr's extremal construction, then it either contains a red copy of $T$ or a blue copy of $K_{m_1, \dots, m_k}$

\begin{lemma} \label{L9}
Suppose that we have numbers $n, k,  \Delta, m_1, \dots, m_k \in \N$  with $k \geq 3$, $m_1 \leq m_2 \leq \ldots \leq m_k$,  $n$ large enough relative to $\Delta, k$ and $n\geq  2 (d+1) m_k$ where $d = 4 \cdot 10^{12} \frac{\log^4{n}}{\log{\log{n}}}$.

Let $T$ be a tree with $|T|=n$ and $\Delta(T) \leq \Delta$.
Let $G$ be a graph with $(k-1)(n-1) + m_1$ vertices that has disjoint vertex sets $H_1, \ldots, H_{k-1}$ with $|H_i| \geq 0.9 n$, such that there are no edges between $H_i$ and $H_j$ for all $i \neq j$. If $G^c$ has no $K_{m_1, \dots, m_k}$, then $G$ contains a copy of $T$.
\end{lemma}
\begin{proof}
Fix $m=m_k$ and $r = \lceil 10^3 \log^2{n} \rceil$.
Notice that we have $n\geq 2(d+1)m$ and  $G^c$ has no $K_{m}^k$.
If $T$ has $\geq n/4r$ leaves, then since $n/4r\geq 13 \Delta |K_{m_1, \dots, m_k}|+1$ we are done by Theorem~\ref{Tmanyleaves}. Therefore, by Lemma~\ref{L1}, we may assume that $T$ has  at least $n/4r$ bare paths of length $r$. 

We first need the following claim.
\begin{claim}
There exist $H'_i \subseteq H_i$ with $|H'_i| \geq 0.8n$ such that for all $S \subseteq H'_i$ with $|S| \leq m$, we have $| N_{H'_i}(S) | \geq 5 \Delta |S|$ and for all $S \subseteq H'_i$ with $m \leq |S| \leq 2m$, we have $| N_{H'_i}(S) | \geq 2n / 3 + 10 \Delta m$.
\end{claim}
\begin{proof}
First observe that for each $i$, $H_i^c$ has no copy of $K^2_m$, or else we could take such a copy together with $m$ vertices from each $H_j : j \neq i$, to obtain a $K^k_m$ in $G^c$, a contradiction. Thus for any $S \subseteq H_i$ with $m \leq |S| \leq 2m$ we have $|N_{H_i}(S)| \geq |H_i|-|S|-m\geq |H_i| - 3m  \geq 0.8 n$.

Now for each $i$, choose a maximal $X_i \subseteq H_i$ with $|X_i| \leq m-1$ such that $|N_{H_i}(X_i)| < 5 \Delta |X_i|$, and let $H'_i = H_i \backslash X_i$. Notice that we have $|H'_i| \geq |H_i| - m \geq 0.8n$ as required by the claim. Using $n\geq 2(d+1)m$ and the fact that n is sufficiently large relative to $\Delta$, we have that for any $S \subseteq H'_i$ with $m \leq |S| \leq 2m$
\[ |N_{H'_i}(S)| \geq |N_{H_i}(S)| - m \geq 0.8 n - m \geq \frac{2}{3}n + 10 \Delta m.\]

Finally, suppose for sake of contradiction that there exists $S \subseteq H'_i$ with $|S| \leq m$ such that $|N_{H'_i}(S)| < 5\Delta |S|$. Then we have $|N_{H_i}(X_i \cup S)| < 5 \Delta |X_i \cup S|$ so that $m \leq |X_i \cup S| \leq 2m$ by maximality of $X_i$ and hence
\[ 10 \Delta m \geq 5 \Delta |X_i \cup S| > |N_{H_i}(X_i \cup S)| \geq 0.8 n ,\]
a contradiction to $n\geq 2(d+1)m$ and $n$ being sufficiently large relative to $\Delta$.
\end{proof}

Let $Z = G \backslash \bigcup_{i = 1}^{k-1}{H_i'}$.
Suppose there exists $v \in Z$ and $a \neq b$ such that $d_{H'_a}(v), d_{H'_b}(v) \geq \Delta$. Apply \Cref{L8} to $T$ in order to get a vertex $u$ and two forests $T_a$ and $T_b$ with no edges between them and $|T_a|, |T_b| \leq 2n/3$. 
We think of the trees in the forests $T_a$ and $T_b$ as being rooted at the neighbours of $u$. Let $t_a, t_b \leq \Delta$ be the number of neighbors of $u$ in $T_a$ and $T_b$ respectively. 
Now choose $X_a \subseteq H'_a \cap N(v)$ so that $|X_a| = t_a$ and $X_b \subseteq H'_b \cap N(v)$ so that $|X_b| = b$. We observe that for $i \in \{a,b\}$, for all $S \subseteq H'_i$ with $1 \leq |S| \leq m$, we have
\begin{equation}\label{EqSmallExpansion}
 | N_{H'_i}(S) \backslash X_i | \geq |N_{H'_i}(S)| - |X_i| \geq 5 \Delta |S| - t_i \geq 4 \Delta |S|.
\end{equation}
Because of the claim and (\ref{EqSmallExpansion}), $H'_a$ satisfies the assumptions of \Cref{L3} with $G=H_a$, $M=2n/3$, $t=t_a$, $X=X_a$, and $\{T_{x_1}, \dots, T_{x_{t}}\}$ the collection of trees in the forest $T_a$. Therefore we can apply \Cref{L3} to $H_a$ in order to find a copy of $T_a$ with its trees rooted in $X_a$.
By the same argument, $H_b$ has a copy of $T_b$ with its trees rooted in $X_b$. These copies of $T_a$ and $T_b$  together with the vertex $v$ give a copy of $T$ in $G$, so we are done.

Otherwise, for all $v \in Z$ there exists $i_v$ such that for all $j \neq i_v$, $d_{H'_j}(v) < \Delta$. We partition $G$ into $k-1$ parts via $G_i = H'_i \cup \{v \in Z : i_v = i\}$. Observe that for any $i \neq j$ and $S \subseteq G_i$, we have $|N(S) \cap H'_j| < \Delta |S|$. We claim that therefore $G_i^c$ has no $K^2_m$. Indeed suppose without loss of generality that $S_1$ was a copy of $K^2_m$ in $G_1^c$. Then for $j = 2, \ldots, k-1$, observing that $|H_j' \backslash N(S_1)|\geq |H_j'|- |N(S) \cap H'_j|\geq 0.8n - 2 \Delta m \geq m$, we can choose a set $S_j \subseteq H'_j \backslash N(S_1)$ of size $m$. Then $\bigcup_{i = 1}^{k-1}{S_i}$ is a copy of $K^k_m$ in $G^c$, a contradiction.

Now fix $i$ and observe that since $G_i^c$ has no $K^2_m$, we have that for any $S \subseteq G_i$ with $|S| \geq m$, $|N_{G_i}(S)| \geq |G_i| - |S|-m$. Now choose $Z_i \subseteq G_i$ with $|Z_i| \leq m-1$ maximal so that $|N_{G_i}(Z_i)| < d|Z_i|$ and let $G'_i = G_i \backslash Z_i$. Then we claim that for all $S \subseteq G'_i$ with $|S| \leq m$, $|N_{G'_i}(S)| \geq d|S|$. Indeed, otherwise we would have $|N_{G_i}(Z_i \cup S)| < d|Z_i \cup S|$, so by maximality of $X$ this would imply $m \leq |Z_i \cup S| \leq 2m$. But then
\[ 2dm \geq d | Z_i \cup S | > |N(Z_i \cup S)| \geq n - |Z_i \cup S| -m\geq n - 3m,\]
a contradiction. 

Now let $n'_i = |G'_i|$. If for some $i$, $n'_i \geq n$ then we can apply \Cref{L6} to conclude that $G'_i$ has a copy of $T$. Otherwise we have that $n'_i \leq n-1$ for all $i \in [k-1]$, and therefore using $|G|= (n-1)(k-1)+m_1$ we conclude $\sum_{i=1}^{k-1}{|Z_i|} \geq m_1$. For each $j = 1, \ldots, k-1$, we observe that 
\[ \left | N\left( \bigcup_{i=1}^{k-1}{Z_i} \right) \cap H'_j \right | \leq  |N(Z_j) \cap H'_j | +\sum_{i \neq j}{|N(Z_i ) \cap H'_j | } \leq d m+k \Delta m ,\] 
and hence
\[ \left | H'_j \backslash N \left ( \bigcup_{i=1}^{k-1}{Z_i} \right ) \right | \geq 0.8n - k \Delta m - d m \geq m.\]
Thus for each $j = 1, \ldots, k-1$ we can choose a set $S_j \subseteq H'_j \backslash N\left(\bigcup_{i =1}^{k-1}{Z_i} \right)$ of size $m_{j+1}\leq m$. But then by taking a subset $X \subseteq \bigcup_{i=1}^{k-1}{Z_i}$ of size $m_1$, we obtain that $X \cup \bigcup_{i=1}^{k-1}{S_i} $ is a copy of $K_{m_1, \dots, m_k}$ in $G^c$, a contradiction.
 \end{proof}

We can now complete the case $k \geq 3$ by using either \Cref{LemmaLinkedKmkFreeTree} or \Cref{L9}.

\begin{proof}[Proof of \Cref{TmainKmmmm} for $k \geq 3$]
Fix $m=m_k$, $d = 4 \cdot 10^{12} \frac{\log^4{n}}{\log{\log{n}}}, r = \lceil 10^3 \log^2{n} \rceil$ and $y = \lceil \log{n} \rceil$. We can choose $C_{\Delta,k}$ such that $n$ is sufficiently large relative to $\Delta, k$ and $n \geq 2(d+1)m$. Let $G$ be a graph with $(k-1)(n-1)+m_1$ vertices such that $G^c$ has no copy of $K_{m_1, \dots, m_k}$. 
Notice that in particular $G^c$ has no $K_m^k$.
If $T$ has at least $n/4r \geq 13 \Delta m + 1$ leaves, then by \Cref{L5} we are done. Otherwise, by \Cref{L1} $T$ has at least $n/4r$ disjoint bare paths of length $r$. 

\begin{claim}\label{ClaimFindQsets}
There are disjoint sets $Q'_1, \dots, Q'_{k-1}$ of size $\in [22 yn/r, 23 yn/r]$, and $W'_1, \dots, W'_{k-1}$ of size $\in [20 yn/r, 21 yn/r]$ such that for all $i$,  $W'_i\subseteq Q'_i$, $Q'_i$ $y$-expands into $W'_i$, and there are no edges between $Q'_i$ and $Q'_j$ for $i\neq j$.
\end{claim}
\begin{proof}
Let $q = 23 yn/r$ and $w= 21 yn/r$.
Since $n$ is sufficiently large relative to $k, \Delta$ and $r = \lceil 10^3 \log^2{n} \rceil$ we have $(n-1)(k-1) + m_1 \geq (k-2)(n + 13 \Delta q) + q$. Therefore we can apply \Cref{C2} to conclude that either $G$ contains a copy of $T$ so that we are done, or else there exists a copy of $K^{k-1}_{q}$ in $G^c$. 
Label the parts of $K_q^{k-1}$ by $Q_1, \ldots, Q_{k-1}$. Observe that clearly $Q_i^c$ has no copy of $K^2_m$. For each $i$, let $W_i \subseteq Q_i$ be a set of size $w$. Now choose $X_i \subseteq Q_i$ with $|X_i| \leq m-1$ maximal so that $|N_{Q_i}(X_i) \cap W_i| < y |X_i|$ and 
let $Q'_i = Q_i \backslash X_i$, and 
$W'_i = W_i \backslash X_i$. 
We claim that for all $S \subseteq Q'_i$ with $|S| \leq m$, $|N_{Q'_i}(S) \cap W'_i| \geq y |S|$. Indeed, otherwise we would have $|N_{Q_i}(X_i \cup S) \cap W_i| < y|X_i \cup S|$ so that $m \leq |X_i \cup S| \leq 2m$ by maximality of $X_i$. But then since $Q_i^c$ has no $K^2_m$,
\[ 2ym \geq y |X_i \cup S| >  |N_{Q_i}(X_i \cup S) \cap W_i| \geq w - |X_i \cup S| - m \geq w - 3m, \]
a contradiction to $n\geq 2(d+1)m$. 
Note that since $m \leq y n / r$, we have $|Q'_i| \geq q - m \geq 22 y n / r$ and $|W'_i| \geq w-m \geq 20 y n / r$.  We further conclude that $Q'_i$ $y$-expands into $W'_i$. Indeed, since $Q_i'^c$ does not have $K^2_m$ we have that for any $S \subseteq Q'_i$ with $m \leq |S| < \left \lceil \frac{w}{2y} \right \rceil$,
\[ |N_{Q'_i}(S)\cap W_i'| \geq |W'_i| - |S| - m \geq w - 2m - \frac{w}{2y} \geq \frac{w}{2} \geq y |S|, \]
so the first condition holds. Moreover, since $Q_i'^c$ does not have $K^2_m$ and $\lceil w/2y \rceil \geq m$, the second condition holds as well.
\end{proof}

Now let $M_i=Q_i'\setminus W_i'$ and note that $y n / r \leq |M_i| \leq 3 y n / r$. For $i\neq j$, fix a maximal family $\mathcal P_{i,j}$ of $\leq 8 k n / r$  vertex-disjoint paths of length $\leq 3$ from $M_i$ to $M_j$ internally outside $R_1 = \bigcup_{i=1}^{k-1}{Q'_i}$.
Let $F$ be an auxiliary graph on $[k-1]$ with $ij$ an edge whenever $|\mathcal P_{i,j}| = 8 k n / r$.  
Let $R_2=\bigcup_{i\neq j} P_{i,j}$ and $R=R_1\cup R_2$.  Note that $|R_1| \leq 23 k y n / r$ and $|R_2| \leq 8k^3 n / r$ so that $|R| \leq 24 k y n / r$ (since $y\geq 8k^2$ as a consequence of $n$ being sufficeintly large relative to $k$.) Now let $M_i' = M_i \backslash R_2$ and note that $|M'_i| \leq |M_i| \leq 3 y n / r$ and 
\[  |M'_i| \geq |M_i| - |R_2| \geq y n / r - 8 k^3 n / r \geq 2 n / r \geq m.\]

Note that $|Q_i'| \leq 23 y n / r$, so $160 \log{|Q'_i|} / \log{\log{|Q'_i|}} \leq \log{n} \leq y$ and hence by \Cref{subsetexpand} $(iii)$, we have that $Q_i'$ $160 \log{|Q'_i|} / \log{\log{|Q'_i|}}$-expands into $W'_i$. Moreover
\[ |Q'_i| \geq 22 \frac{yn}{r} \geq 21 \frac{yn}{r} + \frac{n}{r} \geq |W'_i| + 2 \frac{n}{2r},\]
so we may apply \Cref{LemmaExpanderLinked} with $s = n / 2r$. Since
\[ y \leq \frac{|W'_i|}{40 (n / 2r )} \qquad \text{ and } \qquad  4 \left \lceil \frac{\log{|Q'_i|}}{\log{\log{|Q'_i|}}} \right \rceil \leq 4 \left \lceil \frac{\log{n}}{\log{\log{n}}} \right \rceil \leq \frac{y}{k} - 3,\]
we conclude that $(M_i, W_i')$ is a $(n/2r, y/k - 3, y)$-linked system and hence so is $(M'_i, W'_i)$. We now consider two cases depending on whether $F$ is empty or not.

\textbf{Case 1:} Suppose that $F$ is not empty. 
Let $F'$ be the largest connected component of $F$ and let $k' = |F'| + 1$.  Since $F$ is not empty we have $k' \geq 3$. 
%Now since $M_i$ originally had $kp + u + m$ vertices, it still has at least $u + m$ vertices and so we may take a set $U_i \subset M_i$ of size $u$ and let $M'_i = M_i \backslash U_i$ have at least $m$ vertices. 
Let $G' = G \backslash \bigcup_{i \in F'}{(M'_i \cup N(M_i'))}$. 

\textbf{Case 1.1:} Suppose that $|G' | \geq (k - k')(n + 13 \Delta m) + m$.  Then $G'^c$ has no $K^{k-k' + 1}_m$ or else we could take it together with subsets of $M'_i : i \in F'$ of size $m$ to obtain a $K^k_m$ in $G^c$, a contradiction. But then $G'$ contains a copy of $T$ by \Cref{C2}. 

\textbf{Case 1.2:} Suppose that $|G' | < (k - k')(n + 13 \Delta m) + m$. Then since $m=o(n)$, we have
\[ \left| \bigcup_{i \in F'}{M'_i}\cup N(M'_i) \right| > (k - 1)(n-1) + m_1 - (k-k')(n+13 \Delta m ) - m = (k' - 1)(n-1)(1 - o(1) ).\]
So if we let $Z =  \bigcup_{i \in F'}N(M_i')\backslash R$, we obtain
\begin{align*}
|Z| \geq \left| \bigcup_{i \in F'}{N(M'_i)} \right| - |R| &\geq \left| \bigcup_{i \in F'}{M'_i}\cup N(M'_i) \right| - \left| \bigcup_{i \in F'}{M'_i} \right| - |R|\\
&\geq (k'-1)(n-1)(1 - o(1)) - 3 \frac{ k y n }{ r } - 24 \frac{k y n}{r}\\
&\geq 0.99(k'-1)n.
\end{align*}
Moreover, if we let $X = \bigcup_{i \in F'}{M'_i}$ then we claim $(Z \cup X)^c$ has no $K^{k'}_m$. 
Indeed, since $ij \notin E(F)$ for any $i \in F', j \notin F'$, we could take subsets of $M'_i : i \notin F'$ of size $m$, together with a copy of $K^{k'}_m$ in $(Z \cup X)^c$ to obtain a copy of $K^k_m$ in $G$. Since $F'$ is connected, Lemma~\ref{LemmaJoinManyLinkedSets} applied with $d^-= y/k-3$, $d^+=y$, $s=n/2r$, and $k=k'$ implies that $(X, W)$ is a $(n/2r, y, y)$-linked system for $W = R_2\cup \bigcup_{i \in F'}{W'_i}$.
Thus we may apply Lemma~\ref{LemmaLinkedKmkFreeTree}  to conclude that $G$ contains a copy of $T$.

\textbf{Case 2:} Suppose that $F$ is empty. Note that if $ij \notin E(F)$ then we must have no edges between $M'_i \cup ( N(M'_i) \backslash R )$ and $M'_j \cup ( N(M'_j) \backslash R )$ by the maximality of the family of paths $\mathcal{P}_{i,j}$. Thus if we define $H_i = N(M'_i) \backslash R$, then $H_1, \ldots, H_{k-1}$ are disjoint and there are no edges between $H_i$ and $H_j$, for all $i \neq j$. 
Fix some $i\in \{1, \dots, k-1\}$.
Since $|M'_i| \geq m$, we have that $( G \backslash ( N(M'_i) \cup M'_i ) )^c$ does not contain $K^{k-1}_m$ or else we could take it together with a subset of $M'_i$ of size $m$ to obtain a $K^k_m$ in $G^c$, a contradiction. Thus if $|G \backslash (N(M'_i) \cup M'_i) | \geq (k-2)(n + 13 \Delta m) + m$, then $G \backslash (N(M'_i) \cup M'_i)$ has a copy of $T$ by \Cref{C2}, so we are done. Otherwise we have 
\[ |N(M'_i) \cup M'_i| \geq (n-1)(k-1)+m_1- ((k-2)(n+13\Delta m) +m = n - (k-2)(13 \Delta m + 1) + m_1 - m = n(1 - o(1) ),\] 
so that $|N(M'_i)| \geq n(1 - o(1)) - 3 y n / r = n( 1 - o(1) )$
and hence
\[ |H_i| \geq |N(M'_i)| - |R| \geq n(1 - o(1)) - 24 k y n / r \geq 0.9 n.\]
This holds for all $i$, so we can apply \Cref{L9} to conclude that $G$ contains a copy of $T$.
\end{proof}

\section{Concluding Remarks}

In this paper we determined the range in which bounded degree trees are $H$-good, up to logarithmic factors. However, we conjecture that these factors can be removed to obtain the following.

\begin{conjecture} \label{Cmain}
For all $\Delta$ and $k$ there exists a constant $C_{\Delta, k}$ such for any tree $T$ with max degree at most $\Delta$ and any $H$ with $\chi(H) = k$ satisfying $|T| \geq  C_{\Delta, k} |H| $, $T$ is $H$-good.
\end{conjecture}
\noindent Pokrovskiy and Sudakov \cite{PSpaths} showed that \Cref{Cmain} holds for paths, and our \Cref{Tmanyleaves} shows that \Cref{Cmain} holds for trees with linearly (in $|H|$) many leaves.

Finally, we note that \Cref{Cmain} is best possible up to a constant factor. Indeed, consider the graph consisting of $2k-1$ red cliques of size $n-1$, with all other edges blue. It clearly has no red tree $T$ on $n$ vertices and if $m = n$, then it is not hard to see that it has no copy of $K^k_m$. Thus $R(T, K^k_m) \geq (2k-1)(n-1) + 1 > (k-1)(n-1) + m$, so that $T$ is not $K^k_m$-good.

\section*{Appendix}
Our goal will be to prove \Cref{L3}. This is a generalization of Haxell's theorem \cite{Hax01}, and the proof follows the method of Friedman and Pippenger \cite{FP}. The idea is to prove a stronger statement from which \Cref{L3} will follow as a corollary. For this, we will also need a slightly different definition of neighborhood. For a vertex $x$ in a graph $G$, let $\Gamma(x) = N(x)$ be the neighborhood of $x$ and for a set of vertices $S$ in $G$, define $\Gamma(S) = \bigcup_{x \in S}{\Gamma(x)}$. Also, for a tree $T$ rooted at $v$, we define $d_{root}(T) = d_T(v)$.

\begin{lemma} \label{technical}
Let $\Delta, M, t$ and $m$ be given. Let $X = \{x_1, \ldots, x_t\}$ be a set of vertices in a graph $G$. Suppose that we have rooted trees $T_{x_1}, \ldots, T_{x_t}$ satisfying $\sum_{i=1}^t{\left|T_{x_i}\right|} \leq M$ and $\Delta \left(T_{x_i} \right) \leq \Delta$ for all $i$. Suppose that for all $S$ with $m \leq |S| \leq 2m$ we have $|\Gamma(S)| \geq M + 10 \Delta m$, and for $S$ with $|S| \leq m$ we  have
\begin{equation}\label{e1}
|\Gamma(S)\setminus X|\geq 4\Delta|S\setminus X|+ \sum_{x\in S\cap X} \Big(d_{root}\big(T_x\big) + \Delta\Big).
\end{equation}
Then we find disjoint copies of the trees $T_{x_1}, \ldots, T_{x_t}$ in $G$ such that for each $i$, $T_{x_i}$ is rooted at $x_i$. In addition for all $S \subseteq G$ with $|S| \leq m$, we have
\begin{equation}\label{e2}
\left| \Gamma(S)\setminus \big(T_{x_1}\cup\dots\cup T_{x_t}\big) \right| \geq \Delta|S|.
\end{equation}
\end{lemma}

\begin{proof}
The proof is by induction on $\sum_{i=1}^t e(T_{x_i})$. The initial case is when each tree is just a single vertex which holds by embedding $T_{x_i}$ to $x_i$.
Then (\ref{e2}) holds as a consequence of (\ref{e1}).
Now suppose that the lemma holds for all families of trees with $\sum_{i=1}^t e(T_{x_i})<e$ and we have a family with $\sum_{i=1}^t e(T_{x_i})=e>0$. Without loss of generality, we may assume that $e(T_{x_1})\geq 1$. 
Let $r$ be the root of $T_{x_1}$ and $c$ one of its children. 
For every $v \in \Gamma(x_1)$ we define a  set $X^v=X\cup \{v\}$ and a corresponding family of rooted trees $\{T^v_x : x\in X^v\}$ as follows.
Let $T^v_{x_1}$ be the subtree of $T_{x_1}$ rooted at $r$ formed by deleting $c$ and its children. Let $T^v_{v}$ be the subtree of $T_{x_1}$ rooted at $c$ formed by $c$ and its children. For all $x\in X^v-x_1-v$, let $T^v_x = T_x$.

Suppose that there is a vertex $v\in \Gamma(x_1)\setminus X$ such that the set $X^{v}$ together with the family of trees $\{T^{v}_x : x\in X^{v}\}$ satisfy the following for every $C \subseteq G$ with $|C| \leq m$.
\begin{equation*}
|\Gamma(C)\setminus X^v|\geq 4\Delta|C\setminus X^v|+ \sum_{x\in C\cap X^v} \Big(d_{root}\big(T^{v}_x\big) + \Delta\Big).
\end{equation*}
Then, by induction we have an embedding of $T^{v}_{x_1},\dots, T^{v}_{x_t}, T^{v}_v$ into $G$ which satisfies (\ref{e2}). By adding the edge $x_1 v$, we can join the trees $T^{v}_{x_1}$ and $T^{v}_v$ in order to obtain a copy of $T_{x_1}$ rooted at $x_1$. This gives an embedding of $T_{x_1},\dots, T_{x_t}$ into $G$ which satisfies (\ref{e2}).

Otherwise, for every $v\in \Gamma(x_1) \backslash X$, there is a set $C_v$ with $|C_v| \leq m$ and 
\begin{equation}\label{e3}
|\Gamma(C_v)\setminus X^v|\leq  4\Delta|C_v\setminus X^v|+ \sum_{x\in C_v\cap X^v} \Big(d_{root}\big(T^v_x\big) + \Delta\Big)-1.
\end{equation}
Notice that taking $S=\{x_1\}$, (\ref{e1}) implies that $x_1$ has at least one neighbour outside of $X$.
Define a set of vertices $S$ to be \emph{critical} if it has order  $\leq m$ and equality holds in (\ref{e1}).
\begin{claim}\label{ClaimCvCritical}
For every $v\in \Gamma(x_1) \backslash X$, the set $C_v$ is critical, and also $v\in \Gamma(C_v)$ and $x_1\not\in C_v$.
\end{claim}
\begin{proof}
Notice that the following hold.
\begin{align}
|\Gamma(C_v)\setminus X|-1&\leq |\Gamma(C_v)\setminus X^v|,\label{LHSinequality}\\
4\Delta|C_v\setminus X^v|+ \sum_{x\in C_v\cap X^v} \Big(d_{root}\big(T^v(x)\big) + \Delta\Big)&\leq 4\Delta|C_v\setminus X|+ \sum_{x\in C_v\cap X} \Big(d_{root}\big(T(x)\big) + \Delta\Big).\label{RHSinequality}
\end{align}
 Adding   (\ref{e3}), (\ref{LHSinequality}),  (\ref{RHSinequality}), and (\ref{e1}) applied with $S=C_v$ gives ``$0\leq 0$'' which implies that equality holds in each of these inequalities. In particular equality holds in~(\ref{e1}), which implies that $C_v$ is critical.
For equality in (\ref{LHSinequality}) to hold, we must have $v\in \Gamma(C_v)$. For equality in (\ref{RHSinequality}) to hold, we must have $x_1\not\in C_v$ (since $d_{root}\big(T^v_{x_1}\big)=d_{root}\big(T_{x_1}\big)-1$.)
\end{proof}
We remark that the above proof also gives $v\not\in C_v$, although this will not be needed in the proof.
We'll also need the following claim.
\begin{claim}\label{CriticalUnion}
For two critical sets $S$ and $T$, the union $S\cup T$ is critical.
\end{claim}
\begin{proof}
First we show that the reverse of the inequality (\ref{e1}) holds for $S\cup T$.
We have the following
\begin{equation}\label{es}
|\Gamma(S)\setminus X|= 4\Delta|S\setminus X|+ \sum_{x\in S\cap X} \Big(d_{root}\big(T(x)\big) + \Delta\Big).
\end{equation}
\begin{equation}\label{et}
|\Gamma(T)\setminus X|= 4\Delta|T\setminus X|+ \sum_{x\in T\cap X} \Big(d_{root}\big(T(x)\big) + \Delta\Big).
\end{equation}
\begin{equation}\label{esit}
|\Gamma(S\cap T)\setminus X|\geq 4\Delta|S\cap T\setminus X|+ \sum_{x\in S\cap T\cap X} \Big(d_{root}\big(T(x)\big) + \Delta\Big).
\end{equation}
Equations (\ref{es}) and (\ref{et}) come from $S$ and $T$ being critical, whereas (\ref{esit}) is just (\ref{e1}) applied to $S\cap T$ (which is smaller than $m$ since $S$ is critical.)
Also, note that by inclusion-exclusion, we have
\begin{align}
|S\cup T\setminus X|&= |S\setminus X|+ |T\setminus X|-|S\cap T\setminus X|, \label{ea2}\\
\begin{split}
\sum_{x\in (S\cup T)\cap X} \Big(d_{root}\big(T(x)\big) +  \Delta\Big)&= \sum_{x\in S\cap X} \Big(d_{root}\big(T(x)\big) + \Delta\Big)+\sum_{x\in T\cap X} \Big(d_{root}\big(T(x)\big) + \Delta\Big)\\
& \ \hspace{4cm}-\sum_{x\in S\cap T\cap X} \Big(d_{root}\big(T(x)\big) + \Delta\Big).
\end{split}\label{ea3}
\end{align}
Moreover, we observe that
\begin{align*}
| \Gamma(S \cup T) \backslash X| &= |( \Gamma(S) \cup \Gamma(T) ) \backslash X |, \\
\begin{split}
| \Gamma(S \cap T) \backslash X| &\leq |( \Gamma(S) \cap \Gamma(T) ) \backslash X |, 
\end{split}
\end{align*}
which together with inclusion-exclusion implies
\begin{equation}
| \Gamma(S \cup T) \backslash X| \leq |\Gamma(S) \backslash X| + |\Gamma(T) \backslash X| - |\Gamma(S \cap T) \backslash X|. \label{ea1}
\end{equation}

Plugging (\ref{es}), (\ref{et}), and (\ref{esit}) into (\ref{ea1}), and then using (\ref{ea2}) and (\ref{ea3}) gives 
\begin{equation}\label{el}
|\Gamma(S\cup T)\setminus X|\leq 4\Delta|S\cup T\setminus X|+ \sum_{x\in (S\cup T)\cap X} \Big(d_{root}\big(T(x)\big) + \Delta\Big).
\end{equation}
Since both $S$ and $T$ are critical we have $|S\cup T|\leq 2m$, which together with (\ref{el}) implies that $|\Gamma(S\cup T)|\leq |X|+|\Gamma(S\cup T)\setminus X|\leq |X|+8\Delta m <M+10\Delta m$. By the assumption of the lemma we have $|S\cup T|\leq m$.
Therefore (\ref{e1}) holds for the set $S\cup T$ which together with (\ref{el}) implies that $S\cup T$ is critical. 
\end{proof}

Let $C=\bigcup_{v\in \Gamma(x_1) \backslash X} C_v$. By the two claims, we have that $C$ is critical. Since from the first claim $\Gamma(x_1) \backslash X \subseteq \Gamma(C)$ and $x_1\not\in C$, we have that 
\begin{align*}
|\Gamma(C\cup\{x_1\})\setminus X|&=|\Gamma(C)\setminus X|\\
&= 4\Delta|C\setminus X|+ \sum_{x\in C\cap X} \Big(d_{root}\big(T(x)\big) + \Delta\Big)\\
&< 4\Delta|C\setminus X|+ \sum_{x\in C\cap X} \Big(d_{root}\big(T(x)\big) + \Delta\Big) + d_{root}(T(x_1))+\Delta\\
&= 4\Delta|(C\cup\{x_1\})\setminus X|+ \sum_{x\in (C\cup\{x_1\})\cap X} \Big(d_{root}\big(T(x)\big) + \Delta\Big).
\end{align*}
By (\ref{e1}) we have that $|C\cup\{x_1\}|>m$, which combined with $C$ being critical means that $|C\cup\{x_1\}|=m+1$. But then $|\Gamma(C\cup\{x_1\})|\leq |X|+|\Gamma(C\cup\{x_1\})\setminus X|\leq |X|+8\Delta m$ contradicts the assumption of the lemma that $|\Gamma(C\cup\{x_1\})|\geq M+10\Delta m$.
\end{proof}

\begin{proof}[Proof of \Cref{L3}] 
Note that since $|\Gamma(S)| \geq |N(S)|$ and $\sum_{x \in S \cap X}{\left( d_{root}(T_x) + \Delta \right)} \leq 4 \Delta | S \cap X|$ for all $S$, we may apply \Cref{technical} to obtain copies of $T_{x_1}, \ldots, T_{x_t}$ rooted at $x_1, \ldots, x_t$ respectively so that (\ref{e2}) holds for all $S$ with $|S| \leq m$. In particular, if $S \subseteq T_{x_1} \cup \ldots \cup T_{x_t}$ and $|S| \leq m$ then
\[ \left | N(S) \backslash \left( T_{x_1} \cup \ldots \cup T_{x_t} \right) \right | = \left| \Gamma(S) \backslash \left( T_{x_1} \cup \ldots \cup T_{x_t} \right) \right | \geq \Delta |S|. \qedhere \]
\end{proof}

\end{document}